\documentclass[a4paper,12pt,reqno,twoside]{amsart}

\usepackage{amssymb}
\usepackage{amsmath}
\usepackage{graphicx}
\usepackage[colorlinks,citecolor=blue,urlcolor=blue]{hyperref}
\usepackage{cite}
\usepackage[hmargin=2.5cm,vmargin=2.5cm]{geometry}

\newcommand{\be}{\begin{eqnarray}}
\newcommand{\ee}{\end{eqnarray}}

\newcommand{\beq}{\begin{equation}}
\newcommand{\eeq}{\end{equation}}
\newcommand{\beqn}{\begin{equation*}}
\newcommand{\eeqn}{\end{equation*}}

\newcommand{\average}[1]{\langle#1\rangle}
\newcommand{\slot}{\,\cdot\,}
\newcommand{\round}[1]{\lfloor#1\rfloor}

\newtheorem{thm}{Theorem}

\newtheorem{lem}[thm]{Lemma}

\newtheorem{remark}[thm]{Remark}

\newcommand\cE{{\mathcal E}}

\newcommand\cL{{\mathcal L}}

\newcommand\cP{{\mathcal P}}

\newcommand\bE{{\mathbb E}}

\newcommand\bP{{\mathbb P}}

\newcommand\bR{{\mathbb R}}

\newcommand\bZ{{\mathbb Z}}

\newcommand\rd{{\mathrm d}}

\newcommand\fa{{\mathfrak a}}

\newcommand{\ve}{\varepsilon}

\begin{document}

\title[A local limit theorem for random walks in balanced environments]{A local limit theorem for random walks in balanced environments}

\author[Mikko Stenlund]{Mikko Stenlund}
\address[Mikko Stenlund]{
Department of Mathematics, University of Rome ``Tor Vergata''\\
Via della Ricerca Scientifica, I-00133 Roma, Italy; Department of Mathematics and Statistics, P.O.\ Box 68, Fin-00014 University of Helsinki, Finland.}
\email{mikko.stenlund@helsinki.fi}
\urladdr{http://www.math.helsinki.fi/mathphys/mikko.html}


%\date{\today. (Version~1.03.) {\bf Please do not circulate!}}
%\date{\today. {\bf Please do not circulate!}}

\begin{abstract}
Central limit theorems for random walks in quenched random environments have attracted plenty of attention in the past years. More recently still, finer local limit theorems --- yielding a Gaussian density multiplied by a highly oscillatory modulating factor --- for such models have been obtained. In the one-dimensional nearest-neighbor case with i.i.d.\ transition probabilities, local limits of uniformly elliptic ballistic walks are now well understood. We complete the picture by proving a similar result for the only recurrent case, namely the balanced one, in which such a walk is diffusive. The method of proof is, out of necessity, entirely different from the ballistic case.
\end{abstract}

\maketitle

%%%%%%%%%%%%%%%%%%%%%%%%%%%%
%%%%%%%%%%%%%%%%%%%%%%%%%%%%

\subsection*{Acknowledgements}
The author is most grateful to Carlangelo Liverani, Gerhard Keller, Stefano Olla and Raghu Varadhan for their interest in the problem. He thanks Universit\`a di Roma ``Tor Vergata'' for its hospitality and the Academy of Finland for funding.

\medskip
\begin{center}
%{\it This paper is dedicated to the memory of Esko Valkeila (1951--2012)}
{\it To Esko Valkeila (1951--2012)}
\end{center}
%%%%%%%%%%%%%%%%%%%%%%%%%%%%%%%%%%%%
%%%%%%%%%%%%%%%%%%%%%%%%%%%%%%%%%%%%
%%%%%%%%%%%%%%%%%%%%%%%%%%%%%%%%%%%%

\section{Introduction}
We consider in this paper nearest-neighbor random walks in fixed environments on~$\bZ$. An environment consists of a family of triplets $(q_k,r_k,p_k)$ of non-negative numbers satisfying $q_k+r_k+p_k=1$, one assigned to each site $k$ of the integer lattice~$\bZ$. A discrete-time random walk in the given environment is obtained as the Markov chain $(X_n)_{n\geq 0}$ starting at zero ($X_0=0$) with the time-homogeneous transition probabilities
\beqn
P_{k,j} = \bP(X_{n+1} = j \,|\,X_n = k) = 
\begin{cases}
q_k & \text{if $ j = k-1$},
\\
r_k & \text{if $j = k$},
\\
p_k &\text{if $j = k+1$},
\end{cases}
\eeqn
and $P_{k,j}=0$ if $|k-j|>1$. 

If the triplets $(q_k,r_k,p_k)$ are drawn from the same distribution,  independently of each other, it is known that the corresponding random walk satisfies the central limit theorem in two diametrically opposite regimes. In the first ``diffusive regime'' the walk is assumed to satisfy a sufficiently strong ballisticity condition, in particular so that a nonzero asymptotic speed $\lim_{n\to\infty}X_n/n = v\neq 0$ exists. (In a random environment, this is not an automatic consequence of transience, i.e., of $\lim_{n\to\infty}X_n\in\{\pm\infty\}$.) In the other diffusive regime the environment satisfies \emph{exactly} the conditions
\beq\label{eq:balanced}
q_k = p_k = \omega_k \quad \text{and} \quad r_k = 1-2\omega_k
\eeq
for some numbers $\omega_k\in (0,\tfrac12]$, $k\in\bZ$.
Such an environment, or the resulting walk, is called balanced. In particular, a balanced walk is recurrent, and in fact it is a martingale. In addition to these two, there is a regime of ballistic walks which is sub-diffusive in the sense that the central limit theorem \emph{does not hold}. Moreover, the transient but non-ballistic regime, and the recurrent but non-balanced regime are both sub-diffusive. We refer the reader to \cite{Goldsheid_2007,BolthausenGoldsheid_2008,Peterson_2008} for the precise conditions and the complete picture.

Knowing the diffusive regimes above, one is led to ask if it is possible to do better than the central limit theorem and to obtain a more detailed description of the limit behavior of the random walk than what is provided by the weak convergence following a diffusive scaling. More precisely, one wonders what happens to the site-wise probability mass function $k\mapsto \bP(X_n=k\,|\,X_0=0)$ as $n$ becomes large. Results in this direction are known as local limit theorems. 

For the diffusive ballistic case, a rigorous proof of the local limit theorem proceeds via careful analysis of first hitting times of the walk to various sites of the integer lattice: One proves a local limit theorem for the hitting times and, with the aid of ballisticity, transforms the result concerning hitting times into one concerning the walk. This was first done in \cite{LeskelaStenlund_2011} for the ``purely ballistic'' case $q_k = 0$, assuming rather little about $r_k,p_k$. In particular, no i.i.d.\ or uniform ellipticity assumption was placed. By the same strategy, the result was complemented in the preprint~\cite{DolgopyatGoldsheid_2012} to include diffusive ballistic cases for which $q_k\neq 0$ is allowed, in uniformly elliptic ($q_k,r_k,p_k\geq \ve>0$ for all $k\in\bZ$) i.i.d.\ environments. 

For recurrent walks the approach based on first hitting times seems doomed, because the ``maximum process'' and the actual walk have very little in common in the absence of a nonzero asymptotic speed. Consequently, our line of attack is quite different altogether, and of independent interest: We make explicit use of the central limit theorem, and upgrade the associated weak convergence of the random walk to pointwise convergence of a ``reversed'' random walk. To that end, we take advantage of a classical tool in analysis known as a Nash inequality, originally devised for controlling solutions of parabolic partial differential equations. Namely, it turns out that the abovementioned ``reversed'' random walk satisfies a discrete heat equation with random diffusivity. We then prove, with the aid of a Nash inequality, that the solutions of that equation must be rather regular, which in combination with the central limit theorem yields convergence of the (suitably scaled) density to a smooth Gaussian. Interpreted in terms of the original random walk, the result then yields a highly oscillatory modulated Gaussian density (Figure~\ref{fig:Distribution}). It is worth pointing out that this density does not converge by a long shot (even after scaling).

The phenomenology in the present paper is very similar to the one observed for ballistic walks in the sequence of papers~\cite{SimulaStenlund_2009,SimulaStenlund_2010,LeskelaStenlund_2011}. In fact, also recurrent walks, including balanced ones, were studied in~\cite{SimulaStenlund_2010}. Here we confirm some findings of that paper rigorously.

\section{Preliminaries}
\noindent \textit{Notation.}
Given a function $u$ on $\bZ$, 
let $\nabla u$ denote its discrete gradient
\beqn
\nabla u(k) = u(k+1) - u(k)
\eeqn
and $\Delta u$ its discrete Laplacian
\beqn
\Delta u(k) = u(k+1) - 2u(k) + u(k-1)
\eeqn
throughout this note. For any real number $\xi$, we write $\round{\xi}$ for the largest integer $\leq\xi$.

\medskip
Assume that the numbers $\omega_k\in (0,\tfrac12]$, $k\in\bZ$, are given, and consider the discrete-time random walk $(X_n)_{n\geq 0}$ on $\bZ$, starting at zero, with the balanced environment satisfying~\eqref{eq:balanced}.
%\beqn
%q_k = p_k = \omega_k \quad \text{and} \quad r_k = 1-2\omega_k.
%\eeqn
%\beqn
%P_{k,j} = \bP(X_{n+1} = j \,|\,X_n = k) = 
%\begin{cases}
%\omega_k, & \text{if $|k-j| = 1$},
%\\
%1-2\omega_k, & \text{if $j = k$},
%\\
%0 &\text{otherwise}.
%\end{cases}
%\eeqn
%Such an environment is called balanced.
Next, denote by $\pi$ the measure on $\bZ$ with
\beqn
\pi_k = \frac{1}{\omega_k}.
\eeqn
It is easily checked that $\pi$ is reversible for the Markov transition matrix $P$:
\beqn
\pi_k P_{k,j} = \pi_j P_{j,k}\qquad (j,k)\in \bZ^2.
\eeqn
We are interested in the $k$-dependence of the quantity $(P^{n})_{0,k} =  \bP(X_{n} = k \,|\,X_0 = 0) $ for large values of $n$. In fact, it will be more convenient to study the ``reversed'' quantity
\beqn
a(n,k) = (P^{n})_{k,0} = \frac{\omega_k}{\omega_0} (P^{n})_{0,k}.
\eeqn
This is related to the fact that
\beqn
a(n+1,k)-a(n,k) = ((P-I)P^{n})_{k,0} = \sum_{j\in\bZ} (P-I)_{k,j} a(n,j),
\eeqn
or
\beq\label{eq:heat_discrete}
a(n+1,k)-a(n,k) = \omega_k\Delta a(n,k).
\eeq
In other words, $a$ satisfies a discrete heat equation with variable diffusivity. The initial condition is $a(0,k) = 1_{\{0\}}(k)$ by definition. 

\begin{figure}[!ht]
\begin{centering}
\includegraphics[width=0.7\linewidth]{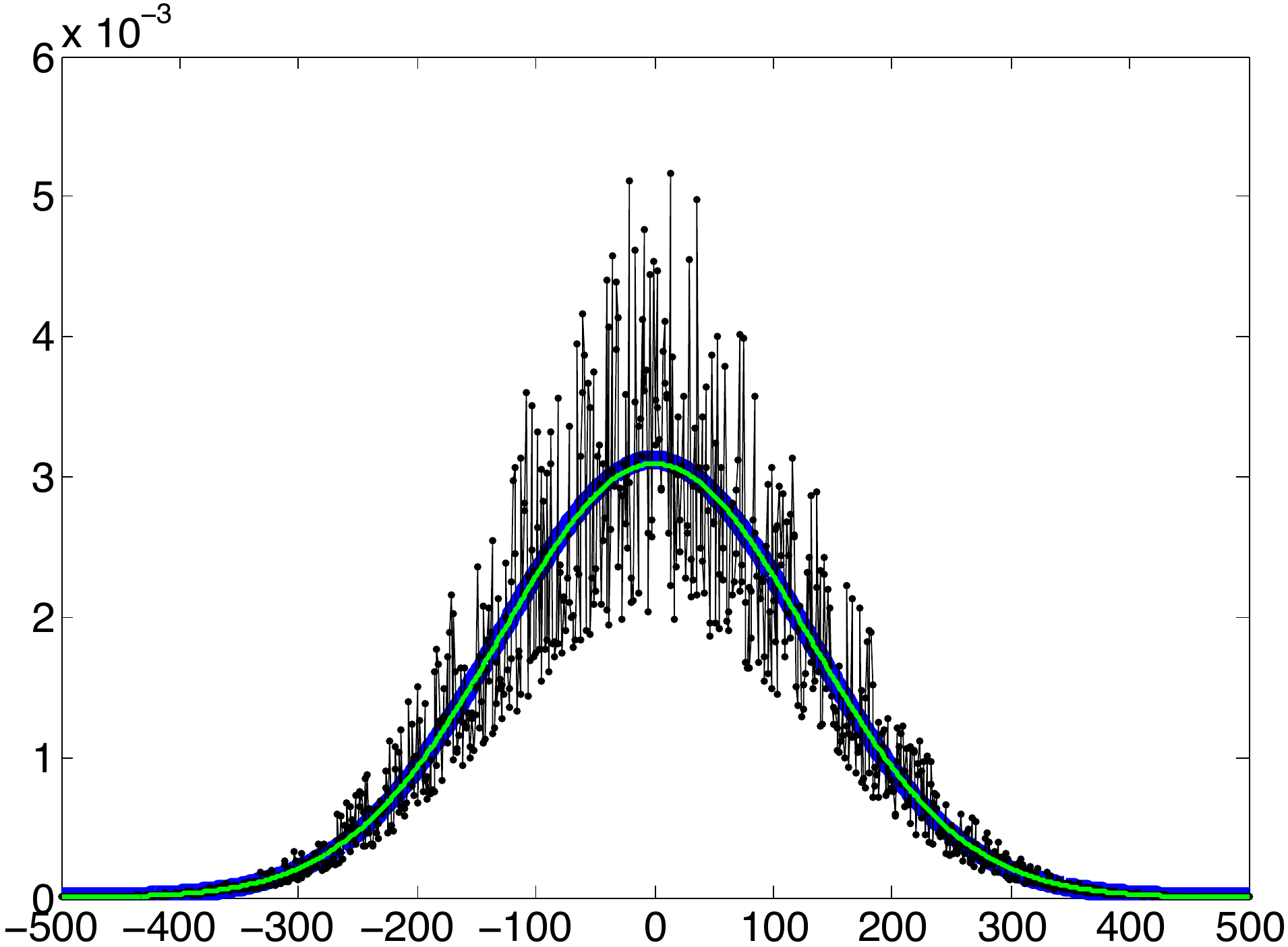}
\caption{The distribution $k\mapsto \bP^0(X_n=k)$ at $n=2^{15}$ in a fixed realization of a balanced random environment is shown in black. The Gaussian density corresponding to the variance of the data is shown in green, whereas the blue curve graphs a solution $k\mapsto \text{const}\cdot a(n,k)$ to the heat equation~\eqref{eq:heat_discrete} with variable diffusivity. The latter two match to a high precision.}
\label{fig:Distribution}
\end{centering}
\end{figure}

It is apparent from Figure~\ref{fig:Distribution} that, as functions of $k$, the ``forward'' and ``reversed'' quantities $(P^n)_{0,k}$ and $(P^n)_{k,0}$ behave very differently at the local level, although their global structures are reminiscent. This observation is at the heart of the present paper.

An intimately related model is the continuous-time random walk $(Y_t)_{t\geq 0}$ on $\bZ$ obtained by Poissonizing the discrete walk with an average of one transition\footnote{The transition could be redundant; when the exponential clock rings so as to allow for a jump, the walker makes a transition from its current location according to the discrete-time transition matrix $P$.} per time unit: Set
\beqn
(\cP^t)_{k,j}=\bP(Y_{s+t} = j\, |\, Y_s = k) = e^{-t}\sum_{n = 0}^\infty \frac{t^n}{n!} (P^n)_{k,j}.
\eeqn
As in the discrete-time case, $\pi$ is reversible for the semigroup $\cP^t$. Moreover, defining
\beq\label{eq:a_def}
\fa(t,k) = (\cP^t)_{k,0},
\eeq
we likewise get the heat equation
\beq\label{eq:heat}
\begin{split}
&\partial_t \fa(t,k)  = \omega_k \Delta \fa(t,k),\quad t\geq 0,
\\
&\fa(0,k) = 1_{\{0\}}(k),
\end{split}
\eeq
with variable diffusivity for $\fa=\fa(t,k)$. We will first concentrate on the continuous-time model, because it makes the analysis technically fluent and the arguments easy to follow. We then discuss in Section~\ref{sec:discrete} how the analysis can be adapted to the discrete-time setting. 

\section{Main results}

Given a number $\sigma>0$, we write
\beqn
\phi_{\sigma^2}(x) = \frac{1}{\sqrt{2\pi}\sigma}e^{-x^2/2\sigma^2}
\eeqn
for the density of the centered normal random variable with variance~$\sigma^2$.

\begin{thm}\label{thm:LLT}
Let $\omega_k\in(0,\tfrac12]$, $k\in\bZ$, be given and assume (A1)--(A3) below:

\medskip
\noindent (A1) There exists a $\sigma>0$ for which
\beq\label{eq:CLT}
\lim_{t\to\infty}\bP^0\bigl(Y_{t}/\sqrt t \leq  x \bigr) = \int_{-\infty}^x \phi_{\sigma^2}(\xi)\,\rd\xi, \quad x\in\bR.
\eeq
That is, the central limit theorem holds for $Y_t/\sqrt t$. (Here $\bP^0$ means $Y_0=0$ is given.)

\medskip
\noindent (A2)
There exists a $\mu>0$ for which
\beq\label{eq:average}
\lim_{T\to\infty} \frac{1}{y-x} \int_{x}^{y} \frac{1}{\omega_{\round{T\xi}}}\,\rd\xi  = \mu,\quad (x,y)\in\bR^2.
\eeq

\medskip 
\noindent (A3) The bound
\beq\label{eq:heat_kernel}
\sup_{t\geq 0} \sup_{k\in\bZ} \sqrt t \,(\cP^t)_{0,k}<\infty
\eeq
holds.

\medskip
\noindent Then the local limit theorem is satisfied in the form that
\beqn
\lim_{t\to\infty}\sup_{x\in I}\left |\omega_{\round{\sqrt t\,x}}\cdot \sqrt{t}\cdot\bP^0\bigl(Y_{t} = \round{\sqrt t\, x}\bigr) -  \frac{1}{\mu}\,\phi_{\sigma^2}(x)\right | = 0
\eeqn
on any compact set $I\subset\bR$.
\end{thm}

\medskip
Let us briefly comment on Assumptions (A1)--(A3):
\begin{remark}
For any uniformly elliptic environment ($\inf_{k\in\bZ}\omega_k>0$), (A3) is satisfied as we will see in Lemma~\ref{lem:heat} of the next section. If the environment is a typical realization of a uniformly elliptic, stationary and ergodic law, then (A1) holds~\cite{Lawler_1982,PapanicolaouVaradhan_1982}. The underlying reason for this is that the balanced walk is a martingale, i.e.,
\beqn
\bE(Y_t \,|\,Y_r,\,r\leq s) = Y_s,\quad s\leq t.
\eeqn
%Incidentally, (A1) can also be proved using homogenization techniques replacing in~\eqref{eq:heat} the random diffusivity $\omega_k$ with the effective constant $\mu^{-1} = \tfrac12\sigma^2$; this will be the topic of another paper.
Incidentally, (A1) can also be proved using homogenization techniques allowing essentially to replace the random diffusivity $\omega_k$ in~\eqref{eq:heat} with the effective constant $\mu^{-1} = \tfrac12\sigma^2$.
Finally, for a stationary and ergodic law $\lambda$ of the environment $\omega=(\omega_k)_{k\in\bZ}$, Birkhoff's ergodic theorem shows that (A2) holds for $\lambda$-almost-every $\omega$, provided $\int \frac{1}{\omega_0}\,\rd\lambda(\omega)<\infty$. In this case $\mu=\int \frac{1}{\omega_0}\,\rd\lambda(\omega)$.
\end{remark}

\medskip
As we will see next, an analogous result is true in the case of discrete time, which explains the peculiar structure of the distribution of the random walk displayed in Figure~\ref{fig:Distribution} by the rapidly oscillatory modulating factor $\omega_{k}$, which was observed in~\cite{SimulaStenlund_2010}. For sufficiently ballistic walks, see \cite{SimulaStenlund_2009,LeskelaStenlund_2011} and the more recent~\cite{DolgopyatGoldsheid_2012} for similar results. 

\begin{thm}\label{thm:LLT_discrete}
Let $\omega_k\in(0,\tfrac12]$, $k\in\bZ$, be given and assume (A2) of Theorem~\ref{thm:LLT} as well as (A1') and (A3') below:

\medskip
\noindent (A1') There exists a $\sigma>0$ for which
\beqn
\lim_{t\to\infty}\bP^0\bigl(X_{n}/\sqrt n \leq  x \bigr) = \int_{-\infty}^x \phi_{\sigma^2}(\xi)\,\rd\xi, \quad x\in\bR.
\eeqn
That is, the central limit theorem holds for $X_n/\sqrt n$. (Here $\bP^0$ means $X_0=0$ is given.)

\medskip 
\noindent (A3') The bound
\beqn
\sup_{n\geq 0} \sup_{k\in\bZ} \sqrt n \,(P^n)_{0,k}<\infty
\eeqn
holds.

\medskip
\noindent Then the local limit theorem is satisfied in the form that, for
\beqn
g_n(x) = \frac12\!\left\{ \bP^0\bigl(X_{n} = \round{\sqrt n\, x}\bigr) +  \bP^0\bigl(X_{n+1} = \round{\sqrt n\, x}\bigr)\right\}
\eeqn
we have
\beq\label{eq:LLT_discrete1}
\lim_{n\to\infty}\sup_{x\in I}\left |\omega_{\round{\sqrt n\,x}}\cdot \sqrt{n}\cdot g_n(x)-  \frac{1}{\mu}\,\phi_{\sigma^2}(x)\right | = 0
\eeq
on any compact set $I\subset\bR$.

\medskip
In addition, suppose that also the following assumption is satisfied:

\medskip 
\noindent (A4) The bound
\beqn%\label{eq:local_times}
\sup_{n\ge 0}\sqrt n \sup_{N\ge n}\sum_{m=n}^N \left(\bP^0\bigl(X_{2m+2} = 0\bigr) -  \bP^0\bigl(X_{2m+1} = 0\bigr)\right) < \infty
\eeqn
holds.

\medskip
\noindent Then 
\beq\label{eq:LLT_discrete2}
\lim_{n\to\infty}\sup_{x\in I}\left |\omega_{\round{\sqrt n\,x}}\cdot \sqrt{n}\cdot  \bP^0\bigl(X_{n} = \round{\sqrt n\, x}\bigr) -  \frac{1}{\mu}\,\phi_{\sigma^2}(x)\right | = 0
\eeq
on any compact set $I\subset\bR$.
\end{thm}

\medskip
As in continuous time, (A3') is satisfied for uniformly elliptic environments (Lemma~\ref{lem:heat_discrete}), and if the environment is also ergodic, then (A1') and (A2) hold. Assumption~(A4) has to do with aperiodicity of the return times to zero, which is of concern in discrete time. For example, the simple symmetric random walk with $\omega_k = \tfrac12$ for all $k\in\bZ$ is recurrent and satisfies $\bP^0\bigl(X_{2m+1} = 0\bigr)=0$ for all $m\ge 0$, and thus violates both (A4) and \eqref{eq:LLT_discrete2}; on the other hand, \eqref{eq:LLT_discrete1} is satisfied because of the average of two successive time steps in the expression of $g_n(x)$. 
Note that $\bP^0\bigl(X_{2m+2} = 0\bigr) -  \bP^0\bigl(X_{2m+1} = 0\bigr)=\omega_0\average{P^m1_{\{0\}},(P^2-P)P^{m}1_{\{0\}}}_\pi\leq 0$ if the eigenvalues of $P$ are non-negative, so that~(A4) is trivially satisfied.  This is true if the walk is lazy, i.e., $\omega_k\leq \tfrac 14$ for all $k\in\bZ$, because then $P = \frac12(I+\hat P)$ where $\hat P$ --- like $P$ --- is a self-adjoint Markov contraction on $L^2(\pi)$. Simulations suggests that generally, under the other conditions of Theorem~\ref{thm:LLT_discrete}, (A4) holds if $\omega_k\neq \tfrac12$ for at least one $k\in\bZ$. In other words, (A4) could well turn out to be equivalent with the Markov chain being aperiodic, but we do not prove this.

\section{Regularity and boundedness of solutions}

Define the quadratic forms
\beqn
\average{ u,v} = \sum_{k\in \bZ} u(k)v(k)
\quad\text{and}\quad
\average{ u,v}_\pi = \sum_{k\in \bZ} \pi_k\, u(k)v(k),
\eeqn
and~\mbox{$\|\slot\|_{L^2}$} and~\mbox{$\|\slot\|_{L^2(\pi)}$} for the corresponding $L^2$-norms, respectively. The following result controls local variations of the solutions to~\eqref{eq:heat} through a gradient bound.
\begin{lem}\label{lem:apriori}
For any $r>0$, the functions defined in~\eqref{eq:a_def} satisfy
\beqn
\int_r^\infty \|\nabla \fa(t,\cdot)\|_{L^2}^2\,\rd t \leq  \tfrac12\pi_0 \,\fa(2r,0).
\eeqn
Moreover, the map $t\mapsto \|\nabla \fa(t,\cdot)\|_{L^2}^2$ is non-increasing.
\end{lem}
\begin{proof}
By reversibility, $\cP^t$ is self-adjoint with respect to the inner product $\average{\slot,\slot}_\pi$. Therefore,
$
 \|\fa(r,\cdot)\|_{L^2(\pi)}^2 = \average{\fa(r,\cdot),\fa(r,\cdot)}_\pi = \average{1_{\{0\}},\cP^{2r}1_{\{0\}}}_\pi = \pi_0 \,\fa(2r,0).
$
By~\eqref{eq:heat},
\beqn
 \partial_t \average{\fa,\fa}_\pi = 2\average{\partial_t \fa,\fa}_\pi  = 2\average{\omega \Delta \fa,\fa}_\pi
= 2\average{\Delta \fa,\fa} = -2\average{\nabla \fa,\nabla \fa}.
\eeqn
The last identity was obtained by summing by parts. The first claim of the lemma follows by integrating with respect to time. On the other hand,
\beqn
\partial_t \average{\nabla \fa,\nabla \fa} = 2 \average{\nabla \partial_t \fa,\nabla \fa} = \average{\nabla (\omega \Delta \fa),\nabla \fa} = - \average{\omega \Delta \fa, \Delta \fa} =  - \average{\omega \Delta \fa, \omega\Delta \fa}_\pi,
\eeqn
where the rightmost expression is non-positive.
\end{proof}

The value of Lemma~\ref{lem:apriori} becomes apparent next, when we look at the solutions of~\eqref{eq:heat} for increasing values of $t$ at a diffusive scale --- that is, when we study the functions
\beqn
f_t(x) = \sqrt t\, \fa(t,\round{\sqrt t\, x} ).
\eeqn

\begin{lem}\label{lem:heat_kernel}
Suppose $\sup_{t> 0}\sqrt t \, (\cP^t)_{0,0}<\infty$ holds.
Then the functions $f_t$ are asymptotically equicontinuous in the sense that there exists a  $C>0$ such that, for any $\ve>0$,
\beqn
|f_t(x)-f_t(y)| \leq C\ve,
\eeqn
provided that $|x-y|\leq \ve^2$ and $t\geq \ve^{-4}$.
\end{lem}
\begin{proof}
For any pair of real numbers $x<y$, the Schwarz inequality yields 
\beqn
\begin{split}
|f_t(y)-f_t(x)|^2 = t\left(\sum_{k=\round{\sqrt t\, x}}^{\round{\sqrt t\, y}-1} \nabla \fa(t,k) \right)^2 \leq t\,\bigl(\round{\sqrt t\, y}-\round{\sqrt t\, x}\bigr) \|\nabla \fa(t,\cdot)\|_{L^2}^2.
\end{split}
\eeqn
By Lemma~\ref{lem:apriori}, the $L^2$-norm above is non-increasing in $t$. Thus,
\beqn
\frac t2 \,\|\nabla \fa(t,\cdot)\|_{L^2}^2 \leq \int_{\frac t2}^\infty \|\nabla \fa(s,\cdot)\|_{L^2}^2\,\rd s \leq  \tfrac12\pi_0 \,\fa(t,0).
\eeqn
Combining the estimates and recalling that $\sqrt t\,\fa(t,0)=\sqrt t \, (\cP^t)_{0,0}$ is bounded, we get
\beqn
|f_t(y)-f_t(x)|^2 \leq \bigl(\sqrt t\, (y-x) + 1\bigr)\, \pi_0 \,\fa(t,0) \leq C'\bigl((y-x)+\tfrac1{\sqrt t}\bigr).
\eeqn
This implies what was to be shown.
\end{proof}

Of course,~\eqref{eq:heat_kernel} implies the condition of Lemma~\ref{lem:heat_kernel}.
Notice that~\eqref{eq:heat_kernel} holds in a homogeneous environment satisfying $\omega_k \equiv \bar\omega$ for all $k\in\bZ$, for any $\bar\omega\in(0,\tfrac12]$, because the process is then a simple symmetric random walk. Intuitively, ``more diffusivity'' should lead to more rapid decay of $(\cP^t)_{0,k}$ with $t$. Indeed, we will next argue that~\eqref{eq:heat_kernel} remains true also when~$\omega_k$ is not constant but $\omega_k \geq \bar\omega$ is satisfied for all $k\in\bZ$. To this end, we recall from~\cite{CarlenKusuokaStroock_1987} a result concerning reversible Markov semigroups, adapted to our needs. It is based on the early work~\cite{Nash_1958}.
But first, let us introduce the Dirichlet quadratic form
\beq\label{eq:Dirichlet}
\cE(u,v) = \average{\nabla u,\nabla v}.
\eeq
\begin{thm}[Nash~\cite{Nash_1958}; Carlen, Kusuoka and Stroock~\cite{CarlenKusuokaStroock_1987}]\label{thm:Nash}
The following conditions are equivalent for the pair $(\cP^t,\pi)$:

\medskip
\noindent (A) There exists a constant $A>0$ such that the Nash inequality
\beqn
\|u\|^6_{L^2(\pi)} \leq A\,\cE(u,u) \|u\|^4_{L^1(\pi)}, \quad u\in L^2(\pi),
\eeqn
holds.

\medskip
\noindent (B) There exists a constant $B>0$ such that the heat kernel bound
\beqn
\|\cP^t\|_{\cL(L^1(\pi),L^\infty(\pi))} \leq B\,t^{-1/2}, \quad t>0,
\eeqn
holds. (Here the norm is the one of linear operators from $L^1(\pi)$ to $L^\infty(\pi)$).
\end{thm}

Observe that the Dirichlet form has the equivalent expression
$
\cE(u,v) 
%= \average{\nabla u,\nabla v} = -\average{u,\Delta v} 
= \average{u,(I-P)v}_\pi.
$
Nevertheless it does not in our case depend on the transition probabilities~$\omega_k$ (unlike~$P$ and~$\pi$). This fact allows us to carry out, for uniformly elliptic environments, the proof of~\eqref{eq:heat_kernel} alluded to earlier:

\begin{lem}\label{lem:heat}
If $\inf_{k\in\bZ}\omega_k>0$, there exists a constant $D>0$ such that
\beqn
(\cP^t)_{i,j} \leq D\,t^{-1/2}, \quad (i,j)\in\bZ^2,\,t> 0.
\eeqn
In particular,~\eqref{eq:heat_kernel} in Assumption (A3) is satisfied.
\end{lem}
\begin{proof}
Let $\bar\omega = \inf_{k\in\bZ}\omega_k>0$. The semigroup $\bar\cP^t$ corresponding to the homogeneous environment $\bar\omega_k = \bar\omega$, $k\in\bZ$, together with the uniform measure $\bar\pi_k = \bar\omega^{-1}$, satisfies~(B) of Theorem~\ref{thm:Nash}. Therefore, it satisfies~(A) of the same theorem. By uniform ellipticity, the measures $\bar\pi$ and $\pi$ are equivalent. This means that~(A) holds also for the pair~$(\cP^t,\pi)$ corresponding to the environment $\omega_k$, $k\in\bZ$, and hence so does~(B). As $(\cP^t)_{i,j} = \omega_i \average{1_{\{i\}},\cP^t 1_{\{j\}}}_\pi \leq \omega_i \|1_{\{i\}}\|_{L^1(\pi)}\| \cP^t 1_{\{j\}} \|_{L^\infty(\pi)}\leq \omega_i \|1_{\{i\}}\|_{L^1(\pi)} \|1_{\{j\}}\|_{L^1(\pi)} B\,t^{-1/2}$, the proof is complete.
\end{proof}

%%%%%%%%%%%%%%%%%%%%%%%%%%%%%%%%%%%
%%%%%%%%%%%%%%%%%%%%%%%%%%%%%%%%%%%

\section{Proof of Theorem~\ref{thm:LLT}}\label{sec:LLT_proof}

Suppose $x<y$ and notice that
\beqn
\begin{split}
\bP(\sqrt t\,x<Y_t\leq \sqrt t\, y) 
& = \sum_{k = \round{\sqrt t \,x} + 1}^{\round{\sqrt t\,y}} (\cP^t)_{0,k}
= \int_{\round{\sqrt t \,x} + 1}^{\round{\sqrt t\,y}}(\cP^t)_{0,\round{\xi}}\,\rd\xi
\\
& = \int_{\sqrt t \,x}^{\sqrt t\,y}(\cP^t)_{0,\round{\xi}}\,\rd\xi + E(t,x,y) = \int_{x}^{y} \frac{\omega_0}{\omega_{\round{\sqrt t\,\xi}}}\,f_t(\xi)\,\rd\xi + E(t,x,y),
\end{split}
\eeqn
where $|E(t,x,y)| \leq  (\cP^t)_{0,\round{\sqrt t\,y}} + (\cP^t)_{0,\round{\sqrt t\,x}} \leq D'\,t^{-1/2}$ for some constant $D'>0$ by Assumption~(A3). We thus obtain from Assumption~(A1) that
\beq\label{eq:CLT1}
\lim_{t\to\infty} \int_{x}^{y} \frac{\omega_0}{\omega_{\round{\sqrt t\,\xi}}} f_t(\xi)\,\rd\xi  = \int_x^y \phi_{\sigma^2}(\xi)\,\rd\xi.
\eeq

Fix any $\ve>0$. Then, by Lemma~\ref{lem:heat_kernel} (valid by (A3)),
\beq\label{eq:osc}
\max_{[x,y]}  f_t - \min_{[x,y]}  f_t \leq C\ve,
\eeq
provided that $|x-y|\leq\ve^2$ and $t\geq \ve^{-4}$. By~\eqref{eq:average} and~\eqref{eq:CLT1}, there also exists $T=T(\ve,x,y)>0$ such that both
\beq\label{eq:average2}
\left|\frac{1}{y-x} \int_{x}^{y} \frac{1}{\omega_{\round{\sqrt t\,\xi}}}\,\rd\xi - \mu\right| \leq \ve
\eeq
and
\beq\label{eq:CLT2}
\frac{1}{y-x}\left| \int_x^y \frac{\omega_0}{\omega_{\round{\sqrt t\,\xi}}} f_t(\xi)\,\rd\xi  - \int_x^y \phi_{\sigma^2}(\xi)\,\rd\xi \right| \leq \ve
\eeq
hold, provided that $t\geq T$. The bound in~\eqref{eq:osc} yields
\beqn
\left|\int_{x}^{y} \frac{1}{\omega_{\round{\sqrt t\,\xi}}}\,f_t(\xi)\,\rd\xi - \int_{x}^{y} \frac{1}{\omega_{\round{\sqrt t\,\xi}}}\,\rd\xi\cdot f_t(x)\right| \leq C\ve\int_{x}^{y} \frac{1}{\omega_{\round{\sqrt t\,\xi}}}\,\rd\xi,
\eeqn
which in combination with~\eqref{eq:average2} results in
\beqn
\begin{split}
& \left|\int_{x}^{y} \frac{1}{\omega_{\round{\sqrt t\,\xi}}}\,f_t(\xi)\,\rd\xi - \mu(y-x)f_t(x)\right|
\\
&\qquad\qquad \leq 
C\ve\int_{x}^{y} \frac{1}{\omega_{\round{\sqrt t\,\xi}}}\,\rd\xi
+
\left| \int_{x}^{y} \frac{1}{\omega_{\round{\sqrt t\,\xi}}}\,\rd\xi  - \mu(y-x)\right| f_t(x) 
%\\
%&\qquad\qquad \leq 
%C\ve\int_{x}^{y} \frac{1}{\omega_{\round{\sqrt t\,\xi}}}\,\rd\xi
%+
%\left|\frac{1}{y-x} \int_{x}^{y} \frac{1}{\omega_{\round{\sqrt t\,\xi}}}\,\rd\xi - \mu\right|  (y-x) f_t(x)
\\
&\qquad\qquad \leq 
C\ve(y-x)\mu
+
\left|\frac{1}{y-x} \int_{x}^{y} \frac{1}{\omega_{\round{\sqrt t\,\xi}}}\,\rd\xi - \mu\right|  (y-x)\left(f_t(x) +C\ve\right)
\\
&\qquad\qquad \leq 
C\ve(y-x)\mu
+
\ve  (y-x)\left(f_t(x) +C\ve\right).
\end{split} 
\eeqn
As $f_t(x)=\sqrt t\,(\cP^t)_{0,\round{\sqrt t\,x}} \omega_{\round{\sqrt t\,x}}/\omega_0$ is uniformly bounded by~(A3), and because~\eqref{eq:CLT2} holds, we conclude that there exists a uniform constant $\bar C>0$ such that, for an arbitrary $\ve>0$,
\beqn
\begin{split}
& \left|\frac{1}{y-x} \int_{x}^{y} \,\phi_{\sigma^2}(\xi)\,\rd\xi - \omega_0 \mu\, f_t(x)\right|
%\\
%&\qquad\qquad
\leq 
\ve + C\ve\mu
+
\ve \left(f_t(x) +C\ve\right) \leq \bar C\ve,
\end{split} 
\eeqn
provided that $y-x\leq \ve^2$ and that $t$ is sufficiently large. It now follows that
\beq\label{eq:limit}
\lim_{t\to\infty} f_t(x) = \frac{1}{\omega_0\mu}\,\phi_{\sigma^2}(x) \equiv f(x),
\eeq
for all $x\in\bR$, because $\phi_{\sigma^2}$ is smooth.

Next, we show that the convergence in~\eqref{eq:limit} is uniform on compact subsets of $\bR$. For that, Lemma~\ref{lem:heat_kernel} is instrumental. Thus, suppose $I\subset\bR$ is an arbitrary closed interval and that $\ve>0$ is given. Since $I$ is compact, we can find a finite subset $\{\xi_i\}_{i=1}^N\subset I$ and a number $t_0>0$ such that, for every $x\in I$, there exists $i=i(x)$ with the property that
\beqn
|f_t(x)-f_t(\xi_i)|\leq \ve/3\quad\text{and}\quad |f(x)-f(\xi_i)|\leq \ve/3
\eeqn
for all $t\geq t_0$. Here we used the equicontinuity guaranteed by Lemma~\ref{lem:heat_kernel}. On the other hand, since $N$ is finite, it follows from~\eqref{eq:limit} that there exists a positive number $t_1>0$,
\beqn
\max_{1\leq i\leq N} |f_t(\xi_i)-f(\xi_i)|\leq \ve/3
\eeqn
for all $t\geq t_1$. Hence,
\beqn
\sup_{x\in I} |f_t(x)-f(x)| \leq \ve, \quad t\geq \max(t_0,t_1).
\eeqn
Since $I$ and $\ve$ were arbitrary, $f_t$ indeed converges to $f$ uniformly on compact sets. \qed %This completes the proof of Theorem~\ref{thm:LLT}.\qed

%%%%%%%%%%%%%%%%%%%%%%%%%%%%%%%%%%%%

\section{Discrete time}\label{sec:discrete}
In this section we will use repeatedly the fact that the operator $P$ is self-adjoint in~$L^2(\pi)$.

\medskip
Define
\beqn
b(n,\cdot) = \frac{a(n+1,\cdot)+a(n,\cdot)}{2} \quad\text{with}\quad b(0,\cdot) = \frac{a(1,\cdot)+a(0,\cdot)}{2}.
\eeqn
The following is a discrete time analogue of Lemma~\ref{lem:apriori}.

\begin{lem}\label{lem:apriori_discrete}
The maps $n\mapsto \|\nabla a(n,\cdot)\|_{L^2}^2$ and $n\mapsto \|\nabla b(n,\cdot)\|_{L^2}^2$ are non-increasing. Moreover, for any $n>0$,
\beqn
\begin{split}
\sum_{m=n}^\infty \|\nabla a(m,\cdot)\|_{L^2}^2\leq  
\pi_0\sup_{N\ge n}\sum_{m=n}^N (a(2m+2,0)-a(2m+1,0))
+ \pi_0\, a(2n,0)
\end{split}
\eeqn
and
\beqn
\sum_{m=n}^\infty \|\nabla b(m,\cdot)\|_{L^2}^2 
\le  \pi_0\, a(2n,0).
\eeqn
\end{lem}
\begin{proof}
Note the identity
\beqn
2\average{u,(P-I)u}_\pi = - \average{u,(P-I)^2u}_\pi - \average{u,(I-P^2)u}_\pi, \quad u\in L^2(\pi).
\eeqn
Here
\beq\label{eq:grad_Lap}
\average{u,(P-I)u}_\pi = \average{u,\omega\Delta u}_\pi = \average{u,\Delta u} = -\average{\nabla u,\nabla u},
\eeq
so that
\beq\label{eq:grad}
2\average{\nabla u,\nabla u} = \average{u,(P-I)^2u}_\pi + \average{u,(I-P^2)u}_\pi.
\eeq
From~\eqref{eq:grad} and $a(n+1,\cdot) = P a(n,\cdot)$,
\beqn
\begin{split}
& 2\sum_{m=n}^\infty \|\nabla b(m,\cdot)\|_{L^2}^2
\\
& \qquad = \sum_{m=n}^\infty\average{b(m,\cdot),(P-I)^2b(m,\cdot)}_\pi + \sum_{m=n}^\infty\left(\|b(m,\cdot)\|_{L^2(\pi)}^2-\|b(m+1,\cdot)\|_{L^2(\pi)}^2\right).
\end{split}
\eeqn
The second sum on the right is also telescoping, so that
\beq\label{eq:grad_b_pre}
2\sum_{m=n}^\infty \|\nabla b(m,\cdot)\|_{L^2}^2 \le \sum_{m=n}^\infty\average{b(m,\cdot),(P-I)^2b(m,\cdot)}_\pi + \|b(n,\cdot)\|_{L^2(\pi)}^2.
\eeq
Note that 
\beq\label{eq:partial_b}
(P-I) b(n,\cdot) = \frac{a(n+2,\cdot)-a(n,\cdot)}{2} = \frac12 (P^2-I)a(n,\cdot).
\eeq
As is easily checked, $P$ is a contraction in $L^1(\pi)$ and $L^\infty(\pi)$. It follows from the Riesz--Thorin interpolation theorem that $P$ is a contraction in $L^2(\pi)$ as well.  In particular, $\average{a(n,\cdot),P^4 a(n,\cdot)}_\pi \le \average{a(n,\cdot),P^2 a(n,\cdot)}_\pi$, which together with \eqref{eq:partial_b} implies
\beqn
\begin{split}
4\average{b(n,\cdot),(P-I)^2b(n,\cdot)}_\pi 
&
= \average{(P^2-I)a(n,\cdot),(P^2-I)a(n,\cdot)}_\pi
\\
&
=  \average{a(n,\cdot),P^4 a(n,\cdot)}_\pi - 2\average{a(n,\cdot),P^2 a(n,\cdot)}_\pi + \average{a(n,\cdot),a(n,\cdot)}_\pi
\\
&
\le \average{a(n,\cdot),a(n,\cdot)}_\pi - \average{a(n,\cdot),P^2 a(n,\cdot)}_\pi
\\
&
= \average{a(n,\cdot),a(n,\cdot)}_\pi - \average{a(n+1,\cdot),a(n+1,\cdot)}_\pi.
\end{split}
\eeqn
In the last line we used the self-adjointness of $P$ and that $a(n+1,\cdot) = Pa(n,\cdot)$.
The final bound obtained above is again of telescoping form, which in \eqref{eq:grad_b_pre} yields
\beqn
\begin{split}
\sum_{m=n}^\infty \|\nabla b(m,\cdot)\|_{L^2}^2 
& \le  \frac18 \|a(n,\cdot)\|_{L^2(\pi)}^2 + \frac12\|b(n,\cdot)\|_{L^2(\pi)}^2 
\le  \|a(n,\cdot)\|_{L^2(\pi)}^2 =  \pi_0\, a(2n,0),
\end{split}
\eeqn
because, by the contractivity of $P$ in $L^2(\pi)$,
\beqn
\begin{split}
\|b(n,\cdot)\|_{L^2(\pi)}^2 & = \frac14\|Pa(n,\cdot)+a(n,\cdot)\|_{L^2(\pi)}^2 \le \frac14\! \left(\|Pa(n,\cdot)\|_{L^2(\pi)}+\|a(n,\cdot)\|_{L^2(\pi)}\right)^2
\\
& \le \frac14\! \left(\|a(n,\cdot)\|_{L^2(\pi)}+\|a(n,\cdot)\|_{L^2(\pi)}\right)^2 = \|a(n,\cdot)\|_{L^2(\pi)}^2,
\end{split}
\eeqn
and because $\|a(n,\cdot)\|_{L^2(\pi)}^2 = \average{1_{\{0\}},a(2n,\cdot)}_\pi$.

On the other hand, from~\eqref{eq:grad} and $(P-I)^2=2(P^2-P)+(I-P^2)$ we get
\beqn
\begin{split}
\sum_{m=n}^N \|\nabla a(m,\cdot)\|_{L^2}^2 & = \sum_{m=n}^N \average{a(m,\cdot),(P^2-P)a(m,\cdot)}_\pi +  \sum_{m=n}^N\average{a(m,\cdot),(I-P^2)a(m,\cdot)}_\pi,
\end{split}
\eeqn
where the second sum on the right side is telescoping and bounded by $\|a(n,\cdot)\|_{L^2(\pi)}^2$. Taking the supremum over $N\ge n$, we see --- equivalently to the claimed bound --- that 
\beqn
\begin{split}
\sum_{m=n}^\infty \|\nabla a(m,\cdot)\|_{L^2}^2\leq  
\sup_{N\ge n}\sum_{m=n}^N \average{a(m,\cdot),(P^2-P)a(m,\cdot)}_\pi
+ \|a(n,\cdot)\|_{L^2(\pi)}^2.
\end{split}
\eeqn
%Note that from~\eqref{eq:grad_Lap} we get
%\beqn
%\begin{split}
%0\leq \|\nabla a(n,\cdot)\|_{L^2}^2 & = \average{a(n,\cdot),(I-P) a(n,\cdot)}_\pi
% = \pi_0 (P^{2n})_{0,0}-\pi_0 (P^{2n+1})_{0,0}
%\end{split}
%\eeqn
%so that
%\beqn
%(P^{2n+1})_{0,0} \le (P^{2n})_{0,0}.
%\eeqn
%Thus we have the result if $(P^n)_{0,0}$ decays monotonically for large $n$.

Regarding monotonicity, by~\eqref{eq:grad_Lap}, 
$
\average{\nabla b(n,\cdot),\nabla b(n,\cdot)} = -\average{b(n,\cdot),(P-I) b(n,\cdot)}_\pi.
$
After an easy computation combining \eqref{eq:partial_b} with the fact that $a(n+1,\cdot)=Pa(n,\cdot)$, this yields
\beqn
\|\nabla b(n+1,\cdot)\|_{L^2}^2 - \|\nabla b(n,\cdot)\|_{L^2}^2 = -\frac14 \|(P^2-I)a(n,\cdot)\|_{L^2(\pi)}^2 \le 0,
\eeqn
indeed. Another easy computation gives
\beqn
\|\nabla a(n+1,\cdot)\|_{L^2}^2 - \|\nabla a(n,\cdot)\|_{L^2}^2 = -\average{(P-I)a(n,\cdot),(P+I)(P-I)a(n,\cdot)}_\pi.
\eeqn
Note that the eigenvalues of $P+I$ are non-negative, because $P$ is a self-adjoint contraction. Hence, the positive square root $Q = (P+I)^{1/2}$ is defined, and
\beqn
\|\nabla a(n+1,\cdot)\|_{L^2}^2 - \|\nabla a(n,\cdot)\|_{L^2}^2 = -\|Q(P-I)a(n,\cdot)\|_{L^2(\pi)}^2 \le 0,
\eeqn
as well.
\end{proof}

\begin{proof}[Proof of Theorem~\ref{thm:LLT_discrete}]Lemma~\ref{lem:apriori_discrete} can be used to derive the discrete-time analogue of Lemma~\ref{lem:heat_kernel}, which was the key ingredient in the proof of the local limit theorem. Indeed, note that under Assumption~(A3') we have
\beqn
\frac n2 \|\nabla b(n,\cdot)\|_{L^2}^2 \le \sum_{m=\lfloor n/2\rfloor}^\infty  \|\nabla b(m,\cdot)\|_{L^2}^2 \le C n^{-1/2},\quad n\geq 1.
\eeqn
Similarly, under Assumption~(A4),
\beqn
\|\nabla a(n,\cdot)\|_{L^2}^2 \le Cn^{-3/2},\quad n\geq 1.
\eeqn
Clearly the statement of Lemma~\ref{lem:heat_kernel} continues to hold modulo notational differences in discrete time, for $f_n(x)=\sqrt n\, b(n,\round{\sqrt n\, x} )$ and $f_n(x)=\sqrt n\, a(n,\round{\sqrt n\, x} )$, respectively. With minuscule modifications, which we leave to the reader, the proof of Theorem~\ref{thm:LLT} in Section~\ref{sec:LLT_proof} then becomes the proof of Theorem~\ref{thm:LLT_discrete}.
\end{proof}

\medskip
We finish the section with an analogue of Lemma~\ref{lem:heat}. To that end, we introduce the Dirichlet form
\beq\label{eq:Dirichlet_discrete}
\cE_2(u,v) = \average{u,(I-P^2) v}_\pi.
\eeq

\begin{thm}[Nash~\cite{Nash_1958}; Carlen, Kusuoka and Stroock~\cite{CarlenKusuokaStroock_1987}]\label{thm:Nash_discrete}
The following conditions are equivalent for the pair $(P,\pi)$:

\medskip
\noindent (A') There exists a constant $A>0$ such that the Nash inequality
\beqn
\|u\|^6_{L^2(\pi)} \leq A\,\cE_2(u,u) \|u\|^4_{L^1(\pi)}\quad\text{when}\quad  \cE_2(u,u)\le \|u\|_{L^1(\pi)}^2,
\eeqn
holds.

\medskip
\noindent (B') There exists a constant $B>0$ such that the heat kernel bound
\beqn
\|P^n\|_{\cL(L^1(\pi),L^\infty(\pi))} \leq B\,n^{-1/2}, \quad n\geq 1,
\eeqn
holds. (Here the norm is the one of linear operators from $L^1(\pi)$ to $L^\infty(\pi)$).
\end{thm}
This is Theorem~4.1 from~\cite{CarlenKusuokaStroock_1987}. Just note that $0\le P_{i,j}<1 = \omega_j\pi_j$ implies
\beqn
\|Pu\|_{L^\infty(\pi)} \le \sup_i \sum_j P_{i,j}|u(j)| \le \sum_j \omega_j \pi_j|u(j)| \le \frac12\|u\|_{L^1(\pi)}
\eeqn
so that the assumption $\|P\|_{\cL(L^1(\pi),L^\infty(\pi))} < \infty$ in~\cite{CarlenKusuokaStroock_1987} is satisfied. 

\medskip 
We now obtain
\begin{lem}\label{lem:heat_discrete}
If $\inf_{k\in\bZ}\omega_k>0$, there exists a constant $D>0$ such that
\beqn
(P^n)_{i,j} \leq D\,n^{-1/2}, \quad (i,j)\in\bZ^2,\,n\ge 1.
\eeqn
In particular, Assumption (A3') is satisfied.
\end{lem}
\begin{proof}
By~\eqref{eq:Dirichlet_discrete} and~\eqref{eq:grad},
\beqn
\cE_2(u,u) = 2\|\nabla u\|_{L^2}^2 - \average{u,(P-I)^2 u}_\pi = 2\|\nabla u\|_{L^2}^2 - \average{\omega,|\Delta u|^2}.
\eeqn
The last expression shows the dependence on $\omega$ explicitly. In particular, denoting~$\tilde\cE_2$ the Dirichlet form of the simple symmetric random walk with $\tilde \omega_k = \tfrac12$ for all $k\in\bZ$,
\beq\label{eq:Dirichlet_monotone}
\tilde\cE_2(u,u) \le \cE_2(u,u).
\eeq
Now, assume 
$
\cE_2(u,u)\le \|u\|_{L^1(\pi)}^2.
$
Then $\tilde\cE_2(u,u) \le (\sup_k\tilde\omega_k/\omega_k\|u\|_{L^1(\tilde\pi)})^2$ 
or
\beq\label{eq:relaxed}
\tilde\cE_2(u,u) \le \frac14\bar\omega^{-2}\|u\|_{L^1(\tilde\pi)}^2,
\eeq
where $\bar\omega = \inf_{k\in\bZ}\omega_k>0$ and $\tilde\pi_k = \tilde\omega_k^{-1}= 2$. In the homogeneous environment $\tilde\omega$, (B') and hence (A') hold true. Note that $\frac14\bar\omega^{-2}\geq 1$, so that the Nash inequality in (A') is not automatic for the function $u$ satisfying~\eqref{eq:relaxed}. However, we will shortly prove the following improvement to Theorem~\ref{thm:Nash_discrete}:

\medskip
\noindent \textit{Claim.} Suppose (B') holds. Given a $C>0$, there exists $\tilde A>0$ such that
\beqn
\|u\|^6_{L^2(\pi)} \leq \tilde A\,\cE_2(u,u) \|u\|^4_{L^1(\pi)} \quad\text{when}\quad \cE_2(u,u)\le C\|u\|_{L^1(\pi)}^2.
\eeqn

\medskip
Thus, we have $\|u\|^6_{L^2(\tilde \pi)} \leq \tilde A\,\tilde\cE_2(u,u) \|u\|^4_{L^1(\tilde \pi)}$ for some $\tilde A$. \noindent By~\eqref{eq:Dirichlet_monotone}, we may replace~$\tilde\cE_2$ by~$\cE_2$ in this inequality. Since~$\tilde\pi$ and~$\pi$ are equivalent, we have shown that (A') holds for the original environment~$\omega$. Hence, also (B') holds, which completes the proof.

\medskip
\noindent\textit{Proof of Claim.}
We follow the proof of Theorem~4.1 in~\cite{CarlenKusuokaStroock_1987}, supplying details.
Let us denote $u_n = \|P^n u\|_{L^2(\pi)}^2$, $n\ge 0$. Since $\cE_2(P^n u,P^n u) = u_n - u_{n+1}$, we have
\beqn
u_0 = u_n + \sum_{k=0}^{n-1} \cE_2(P^k u,P^k u).
\eeqn
Note that
$
\cE_2(u,u)-\cE_2(Pu,Pu) = \average{u,(I-P^2)^2 u}_\pi \ge 0
$
and that $u_n = \average{u,P^{2n}u}_\pi$ yields $u_n\le \|P^{2n}\|_{\cL(L^1(\pi),L^\infty(\pi))} \|u\|_{L^1(\pi)}^2$. Assuming (B'), we thus get
\beqn
\|u\|_{L^2(\pi)}^2 \le (2n)^{-1/2}B\|u\|_{L^1(\pi)}^2 + n\, \cE_2(u,u), \quad n\ge 1.
\eeqn
We perform a minimization. Set $f(x) = ax^{-1/2}+bx$ for $x>0$ with $a=2^{-1/2}B\|u\|_{L^1(\pi)}^2$ and $b = \cE_2(u,u)$. The global minimum is at 
$
x^* = (a/2b)^{2/3}
$
and $f(x^*) = c\cdot a^{2/3}b^{1/3}$ where $c>0$ is independent of $a,b$. Suppose that $u$ satisfies $\cE_2(u,u)\le C\|u\|_{L^1(\pi)}^2$. Then $a/2b \ge 2^{-3/2}B/C$. We can assume without any loss of generality that $B\ge 2^{3/2}C$, so that $x^*\ge 1$. Then, there exists an integer $n^*\ge 1$ such that $n^*\in[x^*,2x^*]$. Obviously $f(n^*) \le f(2x^*) \le 2f(x^*) = 2c\cdot a^{2/3}b^{1/3}$. This shows that
\beqn
\|u\|_{L^2(\pi)}^2 \le 2c \cdot (2^{-1/2}B\|u\|_{L^1(\pi)}^2)^{2/3}(\cE_2(u,u))^{1/3},
\eeqn
\enlargethispage{5mm}%
which proves the Nash inequality.
\end{proof}

%%%%%%%%%%%%%%%%%%%%%%%%%%%%%%%%%%%%
%%%%%%%%%%%%%%    References    %%%%%%%%%%%%%
%%%%%%%%%%%%%%%%%%%%%%%%%%%%%%%%%%%%

\vskip 1cm

%\newpage

\bibliography{Balanced_LLT}{}
\bibliographystyle{plainurl}

%%%%%%%%%%%%%%%%%%%%%%%%%%%%%%%%%%%%

\vspace*{\fill}

\end{document}